\documentclass[11pt]{article}
\usepackage{amssymb}
\usepackage{amsthm}
\usepackage{amsmath}
\usepackage{fullpage,hyperref}
\usepackage{graphicx,url,amsmath,amsthm,amsfonts,amssymb,subfigure,bbm}
\usepackage{floatflt,subfigure}

\newtheorem{theorem}{Theorem}[section]
\newtheorem{lemma}[theorem]{Lemma}

\newtheorem{definition}[theorem]{Definition}
\newtheorem{observation}[theorem]{Observation}

\newtheorem{conjecture}{Conjecture}

\newcommand{\Ind}{\mathbf{1}}
\newcommand{\jnote}[1]{}
\newcommand{\sep}{\,|\,}

\newcommand{\EE}{\mathcal E}
\newcommand{\Lip}{\mathrm{Lip}}

\newcommand{\dist}{\mathsf{dist}}

\newcommand{\remove}[1]{}
\newcommand{\1}{\mathbf{1}}

\newcommand{\len}{\mathsf{len}}

\newcommand{\e}{\varepsilon}
\newcommand{\KK}[2]{{K_{2,#1}^{\oslash #2}}}

\title{{\bf Coarse differentiation and multi-flows in planar graphs}}

\author{James R. Lee\thanks{University of Washington,
    Seattle.  Research partially supported by NSF CAREER award CCF-0644037.  Part of this research was completed while the author was a postdoctoral fellow at the Institute for Advanced Study, Princeton.} \and Prasad Raghavendra\thanks{University of Washington, Seattle. Research supported in part by NSF grant CCF-0343672.}}
\date{}

\begin{document}

\maketitle

\begin{abstract}
We show that the multi-commodity max-flow/min-cut gap for
series-parallel graphs can be as bad as 2, matching a recent upper
bound \cite{CLV07} for this class, and resolving one side
of a conjecture of Gupta, Newman, Rabinovich, and Sinclair.

 This also improves
the largest known gap for planar graphs from $\frac32$ to $2$,
yielding the first lower bound
that doesn't follow from elementary calculations.
Our approach
uses the {\em coarse differentiation} method
of Eskin, Fisher, and Whyte
in order to lower bound the distortion for
embedding a particular family of shortest-path metrics into $L_1$.
\end{abstract}


\section{Introduction}

Since the appearance of \cite{LLR95} and \cite{AR98},
low-distortion metric embeddings have become
an increasingly powerful tool in studying the relationship between
cuts and multicommodity flows in graphs. For background on
the field of metric embeddings and their applications in theoretical
computer science, we refer to
Matou{\v{s}}ek's book \cite[Ch. 15]{Mat01}, the surveys
\cite{Ind01,LinialSurvey}, and the compendium of open problems
\cite{MatOpen}.

One of the central connections
lies in the correspondence between low-distortion
$L_1$ embeddings, on the one hand,
and the {\em Sparsest Cut problem} (see, e.g. \cite{LLR95,AR98,ARV04,ALN05}) and
{\em concurrent multi-commodity flows} (see, e.g. \cite{GNRS99,CGNRS06})
on the other.
This relationship allows one
to bring sophisticated geometric and analytic
techniques to bear on classical problems in graph partitioning
and in the theory of network flows.
In the present paper, we show how techniques developed initially in
geometric group theory can be used to shed new light on
the connections between sparse cuts and multi-commodity flows
in planar graphs.

\medskip
\noindent
{\bf Multi-commodity flows and sparse cuts.}
Let $G=(V,E)$ be an undirected graph, with a {\em capacity} $C(e)\ge
0$ associated to every edge $e\in E$. Assume that we are given $k$
pairs of vertices $(s_1,t_1),...,(s_k,t_k)\in V\times V$ and
$D_1,\ldots,D_k\ge 1$. We think of the $s_i$ as {\em sources}, the
$t_i$ as {\em targets}, and the value $D_i$ as the {\em demand} of
the {\em terminal pair} $(s_i,t_i)$ for {\em commodity}
$i$.

In the {\em MaxFlow} problem the objective is to maximize the {\em
fraction} $\lambda$ of the demand that can be shipped
simultaneously for all the commodities, subject to the capacity
constraints. Denote this maximum by $\lambda^*$. A straightforward upper
bound on $\lambda^*$ is the {\em sparsest cut ratio}.
Given any subset $S \subseteq V$,
we write
$$
\Phi(S) = \frac{\sum_{uv \in E} C(uv) \cdot |{\bf 1}_S(u) - {\bf 1}_S(v)|}
               {\sum_{i=1}^k D_i \cdot |{\bf 1}_S(s_i) - {\bf 1}_S(t_i)|},
$$
where ${\bf 1}_S$ is the characteristic function of $S$.
The value $\Phi^* = \min_{S \subseteq V} \Phi(S)$
is the minimum over all cuts (partitions) of $V$, of the ratio
between the total capacity crossing the cut and the total demand
crossing the cut. In the case of a single commodity (i.e. $k=1$)
the classical MaxFlow-MinCut theorem states that
$\lambda^*=\Phi^*$, but in general this is no longer the case. It
is known~\cite{LLR95,AR98} that $\Phi^*=O(\log k)\lambda^*$.
This result is perhaps the first striking application of metric
embeddings in combinatorial optimization (specifically, it uses
Bourgain's embedding theorem~\cite{Bourgain85}).

Indeed, the connection between $L_1$ embeddings and multi-commodity
flow/cut gaps can be made quite precise.
For a graph $G$,
let $c_1(G)$ represent the largest distortion necessary to embed any
shortest-path metric on $G$ into $L_1$ (i.e. the maximum over all possible
assignments of non-negative lengths to the edges of $G$).  Then $c_1(G)$
gives an upper bound on the ratio between the sparsest cut ratio and the
maximum flow for any multi-commodity flow instance on $G$ (i.e.
with any choices of capacities and demands) \cite{LLR95,AR98}.
Furthermore, this connection is tight in the sense that there
is always a multi-commodity flow instance on $G$ that achieves a gap of $c_1(G)$ \cite{GNRS99}.

\medskip

Despite significant progress \cite{OS81,GNRS99,CGNRS06,CLV07,BKL06}, some fundamental questions are still left unanswered.
As a prime example, consider the well-known {\em planar embedding conjecture} \cite{GNRS99,Ind01,LinialSurvey,MatOpen}:
\begin{quote}
{\em There exists a constant
$C$ such that every planar graph metric embeds into $L_1$ with distortion
at most $C$.}
\end{quote}

In initiating a systematic study of $L_1$ embeddings \cite{GNRS99}
for minor-closed families, Gupta, Newman, Rabinovich, and Sinclair put forth the following vast
generalization of this conjecture (we refer to \cite{DiestelBook} for the relevant graph theory).

\begin{conjecture}[Minor-closed embedding conjecture]
If $\mathcal F$ is any non-trivial minor-closed family, then $\sup_{G \in \mathcal F} c_1(G) < \infty$.
\end{conjecture}

\noindent
{\bf Lower bounds on the multi-commodity max-flow/min-cut ratio in planar graphs.}
While techniques for proving upper bounds on the $L_1$-distortion required
to embed such families has steadily improved, progress on lower bounds has
been significantly slower, and recent breakthroughs in lower bounds
for $L_1$ embeddings of discrete metric spaces that rely
on discrete Fourier analysis \cite{KV05,KN06} do not apply to excluded-minor
families.

The best previous lower bound on $c_1(G)$ when $G$ is a planar graph occurred
for $G = K_{2,n}$, i.e. the complete $2 \times n$ bipartite graph.
By a straightforward generalization of the lower bound of Okamura and Seymour \cite{OS81},
it is possible to show that $c_1(K_{2,n}) \to \frac32$ as $n \to \infty$
(see also \cite{ADGIR03} for a simple proof of
this fact in the dual setting).

We show that, in fact there is an infinite family of series-parallel (and hence, planar)
graphs $\{G_n\}$ such that $\lim_{n \to \infty} c_1(G_n) = 2$; this is not only
a new lower bound for planar graphs, but yields an {\em optimal} lower bound on the
$L_1$-distortion (and hence the flow/cut gap) for series-parallel graphs.
The matching upper bound was recently proved in \cite{CLV07}.

\subsection{Results and techniques}

%
%
%

Generalizations of classical differentiation theory have played
a prominent role in proving the non-existence of bi-Lipschitz
embeddings between various spaces, when the target space $Z$
is sufficiently nice (e.g. if $Z$ is a Banach space with the Radon-Nikodym property);
see, for instance \cite{pansu,Cheeger99,LN-heisenberg,benlin,CK-Chern}.  But
this approach does not apply to targets like $L_1$ which don't have
the Radon-Nikodym property; in particular, even Lipschitz mappings $f : \mathbb R \to L_1$
are not guaranteed to be differentiable in the classical sense.

More recently, however, Cheeger and Kleiner \cite{CK1-06,CKpub}
have successfully
applied weaker notions of differentiability
to the study of $L_1$ embeddings  of the Heisenberg group.
Their approach reduces the non-embeddability problem to
understanding the structure of sets of finite perimeter,
for which they build on the local regularity theory of \cite{FSSC01}.
See \cite{LN-heisenberg} for the relevance to integrality gaps for the Sparsest Cut problem.

Our lower bounds are also inspired by differentiation theory.
We use the {\em coarse differentiation} technique
of Eskin, Fisher, and Whyte \cite{EFW06}, which gives a discrete approach to
finding local regularity in distorted paths.

\medskip

Our basic approach is simple; we know that $c_1(K_{2,n}) \leq \frac32$ for every $n \geq 1$.
But consider $s,t \in V(K_{2,n})$ which constitute the partition of size 2.
Say that a cut $S \subseteq V(K_{2,n})$ is {\em monotone with respect to $s$ and $t$}
if every simple $s$-$t$ path in $K_{2,n}$ has at most one edge crossing the cut $(S,\bar S)$.
It is not difficult to show that if an $L_1$ embedding is composed entirely of cuts
which are monotone with respect to $s$ and $t$, then that embedding must have distortion
at least $2 - \frac{2}{n}$.

Consider now the recursively defined family of graphs $\KK{n}{k}$, where $\KK{n}{1} = K_{2,n}$
and $\KK{n}{k}$ arises by replacing every edge of $\KK{n}{k-1}$ with a copy of $K_{2,n}$.
The family $\{\KK{2}{k}\}_{k \geq 1}$  are the well-known {\em diamond graphs}
of \cite{NR03,GNRS99}.  We show that in any low-distortion embedding of $\KK{n}{k}$ into $L_1$,
for $k \geq 1$ large enough, it is possible to find a (metric) copy of $K_{2,n}$ for which the
induced embedding is composed almost entirely of monotone cuts.  The claimed distortion
bound follows, i.e. $\lim_{n,k \to \infty} c_1(\KK{n}{k}) = 2$.
In Section \ref{sec:embeddings}, we exhibit embeddings which
show that for every fixed $n$,
$\lim_{k \to \infty} c_1(\KK{n}{k}) < 2$,
thus it is necessary to have the base graphs grow in size.

The ability to find these monotone copies of $K_{2,n}$ inside a low-distortion $L_1$ embedding
of $\KK{n}{k}$ arises from two sources.  The first is the coarse differentiation technique
mentioned earlier; this is carried out in Section \ref{sec:diff}.
The second aspect is the relationship between regularity and monotonicity for $L_1$ embeddings
which is expounded upon in Section \ref{sec:mono}, and relies on the well-known
fact that every $L_1$ embedding decomposes in a certain way into a distribution over cuts.

\medskip

We remark that a similar approach was discovered independently
by Cheeger and Kleiner (in the significantly more sophisticated
setting of the Heisenberg group).
Section 1.8 of \cite{CK1-06} describes an alternate proof of
of the non-embeddability of the Heisenberg group into $L_1$ which uses metric
differentiation in the sense of \cite{K94,P2001}, and a classification
of monotone subsets of the Heisenberg group.  This is carried out in full
detail in \cite{CK2009}.

\section{Preliminaries}
   For a graph $G$, we will use $V(G), E(G)$ to denote the sets of vertices and edges of $G$, respectively.
    Sometimes we will equip $G$ with a non-negative length function $\len : E(G) \to \mathbb R_+$,
    and we let $d_{\len}$ denote the shortest-path (semi-)metric on $G$.  We say that $\len$ is
    a {\em reduced length} if $d_{\len}(u,v) = \len(u,v)$ for every $(u,v) \in E(G)$.
    All length functions considered in the present paper will be reduced.
    We will write $d_G$ for the path metric on $G$ if the length function is implicit.
    For an integer $n$, let $K_{2,n}$ denote the complete bipartite graph with $2$ vertices on one side, and $n$ on the other.

\subsection{$s$-$t$ graphs and $\oslash$-products}

An $s$-$t$ graph $G$ is
a graph which has two distinguished vertices $s,t \in V(G)$.  For an $s$-$t$ graph,
we use $s(G)$ and $t(G)$ to denote the vertices labeled $s$ and $t$, respectively.
Throughout this article, the graphs $K_{2,n}$ are considered $s$-$t$ graphs in the natural way
(the two vertices forming one side of the partition are labeled $s$ and $t$).

\begin{figure}
\centering
\includegraphics[width=12cm]{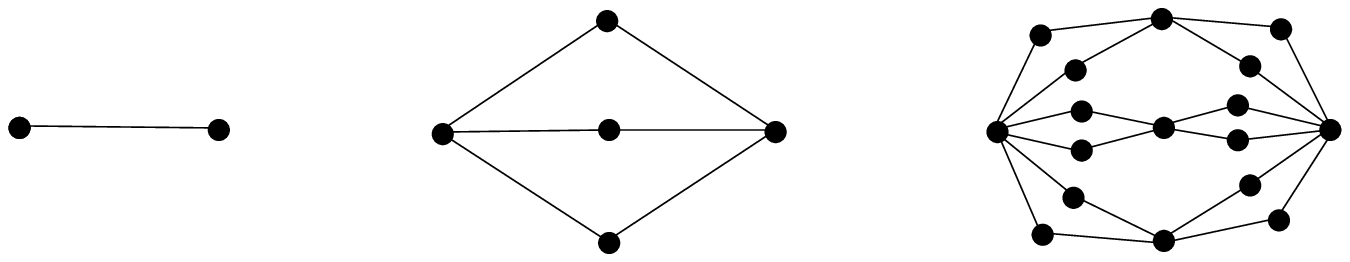}
\caption{A single edge $H$, $H \oslash K_{2,3}$, and $H \oslash K_{2,3} \oslash K_{2,2}$.}
\label{fig:oslash}
\end{figure}

\begin{definition}[Composition of $s$-$t$ graphs]
Given two $s$-$t$ graphs $H$ and $G$, define $H \oslash G$ to be the $s$-$t$
graph obtained by replacing each edge $(u,v) \in E(H)$ by a copy of $G$
(see Figure \ref{fig:oslash}).
Formally,
\begin{itemize}
\item $V(H \oslash G) = V(H) \cup \left( \vphantom{\bigoplus} E(H) \times V(G) \setminus \{s(G),t(G)\} \right).$
\item For every edge $e=(u,v) \in E(H)$, there are $|E(G)|$ edges,
\begin{eqnarray*}
&&
\!\!\!\!\!\!\!\!\!\!\!\!\!\!\!
\left\{ \left(\vphantom{\bigoplus} (e,v_1), (e,v_2) \right)\,|\, (v_1,v_2) \in E(G) \textrm{ and } v_1, v_2 \notin \{s(G),t(G)\}\right\}
\cup
\\
&&  \left\{ \left(\vphantom{\bigoplus} u, (e, w)\right) \,|\, (s(G),w) \in E(G) \right\}
\cup \left\{ \left(\vphantom{\bigoplus}  (e, w),v\right) \,|\, (w,t(G)) \in E(G) \right\}
\end{eqnarray*}
\item $s(H \oslash G) = s(H)$ and $t(H \oslash G) = t(H)$.
\end{itemize}

If $H$ and $G$ are equipped with length functions $\len_H, \len_G$, respectively,
we define $\len_{H \oslash G}$ as follows.  Using the preceding notation,
for every edge $e = (u,v) \in E(H)$,
\begin{eqnarray*}
\len\left((e,v_1),(e,v_2)\right) &=& \frac{\len_H(e)}{d_{\len_G}(s(G),t(G))} \len_G(v_1,v_2) \\
\len\left(u,(e,w)\right) &=& \frac{\len_H(e)}{d_{\len_G}(s(G),t(G))} \len_G(s(G),w) \\
\len\left((e,w),v\right) &=& \frac{\len_H(e)}{d_{\len_G}(s(G),t(G))} \len_G(w,t(G)).
\end{eqnarray*}
This choice implies that $H \oslash G$ contains an isometric copy of $(V(H),d_{\len_H})$.
\end{definition}

Observe that there is some ambiguity in the definition above, as there are two ways
to substitute an edge of $H$ with a copy of $G$, thus we assume
that there exists some arbitrary orientation of the edges of $H$.  However, for our purposes
the graph $G$ will be symmetric, and thus the orientations are irrelevant.

\begin{definition}[Recursive composition]
For an $s$-$t$ graph $G$ and a number $k \in \mathbb N$, we define
$G^{\oslash k}$ inductively by letting $G^{\oslash 0}$ be a single edge of unit length,
and setting $G^{\oslash k} = G^{\oslash k-1} \oslash G$.
\end{definition}

The following result is straightforward.

\begin{lemma}[Associativity of $\oslash$]
For any three graphs $A,B,C$, we have $(A \oslash B) \oslash C = A
\oslash (B \oslash C)$, both graph-theoretically and as metric
spaces.
\end{lemma}

\begin{definition}
   For two graphs $G$, $H$, a subset of vertices $X \subseteq V(H)$ is said to be a {\it copy} of $G$ if there exists a
   bijection $f : V(G) \rightarrow X$ with distortion 1, i.e. $d_{H}(f(u),f(v)) = C\cdot d_{G}(u,v)$ for some constant $C > 0$.
\end{definition}

   Now we make the following two simple observations about copies of $H$ and $G$ in $H \oslash G$.
   \begin{observation}\label{obs1}
   The graph $H \oslash G$ contains $|E(H)|$ distinguished copies of the graph $G$, one copy corresponding to each edge in $H$.
   \end{observation}

   \begin{observation}\label{obs2}
   The subset of vertices $V(H) \subseteq V(H \oslash G)$ form an isometric {\it copy} of $H$.
   \end{observation}

For any graph $G$, we can write $G^{\oslash N} = G^{\oslash k-1}
\oslash G \oslash G^{N-k}$. By observation \ref{obs1}, there are
$|E(G^{\oslash k-1})| = |E(G)|^{k-1}$  copies of $G$ in
$G^{\oslash k-1} \oslash G$.  Now using observation \ref{obs2}, we
obtain $|E(G)|^{k-1}$ copies of $G$ in $G^{\oslash N}$.  We
refer to these as the {\it level-$k$} {\it copies} of $G$, and their
vertices as {\it level-$k$ vertices.}

In the case of $K_{2,n}^{\oslash N}$, we will use a compact notation
to refer to the copies of $K_{2,n}$.  For two level-$k$
vertices $x,y \in V(K_{2,n}^{N})$, we will use $K_{2,n}^{(x,y)}$ to
denote the copy of $K_{2,n}$ for which $x$ and $y$ are the
$s$-$t$ points. Note that such a copy does not exist between all
pairs of level-$k$ vertices.

\subsection{Cuts and $L_1$ embeddings}

\medskip
\noindent
{\bf Cuts.}
A cut of a graph is a partition of $V$ into $(S,\bar{S})$---we sometimes
refer to a subset $S \subseteq V$ as a cut as well. A cut
gives rise to a semi-metric; using indicator functions, we can write the cut
semi-metric as $\rho_S(x,y) = |\Ind_S(x) - \Ind_S(y)|$. A fact central to our
proof is that embeddings of finite metric spaces into $L_1$ are equivalent to sums
of positively weighted cut metrics over that set (for a simple proof
of this see \cite{GeoCuts}). 

A {\em cut measure on $G$} is a function $\mu : 2^V \to \mathbb R_+$
for which $\mu(S) = \mu(\bar S)$ for every $S \subseteq V$.
Every cut measure gives rise to an embedding $f : V\to L_1$ for which
\begin{equation}
\label{eq:cutmeasure}
\|f(u)-f(v)\|_1 = \int |\1_S(u)-\1_S(v)|\,d\mu(S),
\end{equation}
where the integral is over all cuts $(S, \bar S)$.
Conversely, to every embedding $f : V \to L_1$, we can
associate a cut measure $\mu$ such that \eqref{eq:cutmeasure} holds.
We will use this correspondence freely in what follows.
When $V$ is a finite set (as it will be throughout), for $A \subseteq 2^V$, we
define $\mu(A) = \sum_{S \in A} \mu(S)$.

\medskip
\noindent
{\bf Embeddings and distortion.}
If $(X,d_X),(Y,d_Y)$ are metric spaces, and
$f : X \to Y$, then we write $$\|f\|_\Lip = \sup_{x \neq y \in X} \frac{d_Y(f(x),f(y))}{d_X(x,y)}.$$
If $f$ is injective, then the {\em distortion of $f$} is
defined by $\dist(f) = \|f\|_\Lip \cdot \|f^{-1}\|_\Lip$.
A map with distortion $D$ will sometimes be referred to as {\em $D$-bi-lipschitz.}
If $d_Y(f(x),f(y)) \leq d_X(x,y)$ for every $x,y \in X$,
we say that $f$ is {\em non-expansive.}
For a metric space $X$, we use $c_1(X)$ to denote the least distortion required to embed $X$ into $L_1$.

\section{Coarse differentiation}\label{sec:diff}

In the present section, we study the regularity of paths under bi-lipschitz mappings into $L_1.$
Our main tool is based on differentiation \cite{EFW06}.  First, we need a discrete analog of bounded variation.

\begin{definition}
A sequence $\{x_1,x_2,\ldots,x_k\} \subseteq X$ in a metric space $(X,d)$ is said to $\epsilon$-efficient if
$$
d(x_1,x_k) \leq \sum_{i=1}^{k-1} d(x_i, x_{i+1}) \leq (1+ \epsilon)\, d(x_1,x_k)
$$
\end{definition}

Of course the left inequality follows trivially
from the triangle inequality.

\begin{definition}
A function $f : Y \rightarrow X$ between two metric spaces $(X,d)$ and $(Y,d')$,
is said to be {\em $\epsilon$-efficient on ${P} = \{y_1,y_2,\ldots,y_k\} \subseteq Y$}
if the sequence $f({P}) = \{f(y_1), f(y_2), \ldots f(y_k)\}$ is $\epsilon$-efficient in $X$.
\end{definition}

For the sake of simplicity,
we first present the coarse differentiation argument for a function on $[0,1]$.
Let $f : [0,1] \to X$ be a non-expansive map into a metric space $(X,d)$.  Let $M \in \mathbb N$ be given, and for each $k \in \mathbb N$, let $L_k = \{ j M^{-k} \}_{j=0}^{M^k} \subseteq [0,1]$ be the set of {\em level-$k$ points}, and let $S_k = \left\{ (j M^{-k}, (j+1) M^{-k}) : j \in \{1,\ldots,M^k-1\} \right\}$ be the set of {\em level-$k$ pairs}.

For an interval $I = [a,b]$, $f|_{I}$ denotes the restriction of $f$ to the interval $I$.  Now we say that $f|_{I}$ is {\em $\e$-efficient at granularity $M$} if
$$
\sum_{j=0}^{M-1} d\left( f\left(a + \frac{(b-a)j}{M}\right), f\left(a + \frac{(b-a)(j+1)}{M}\right)\right) \leq \left(1+\e\right) d(f(a),f(b)).
$$

Further, we say that a function $f$ is {\em $(\e,\delta)$-inefficient at level $k$} if
$$
\left | \left \{ (a,b) \in S_k : \textrm{$f|_{[a,b]}$ is not $\e$-efficient at granularity $M$} \right\}\right| \geq \delta M^k.
$$
In other words, the probability that a randomly chosen level $k$ restriction  $f|_{[a,b]}$ is
not $\e$-efficient is at least $\delta$.
Otherwise, we say that $f$ is {\em $(\e,\delta)$-efficient at level $k$}.
The main theorem of this section follows.

\begin{theorem}[Coarse differentiation]
\label{thm:coarsediff}
If a non-expansive map $f : [0,1] \to X$ is $(\e,\delta)$-inefficient at an $\alpha$-fraction of levels
$k = 1, 2, \ldots, N$, then
$\mathsf{dist}(f|_{L_{N+1}}) \geq \frac{1}{2} \varepsilon \alpha \delta N$.
\end{theorem}

\begin{proof}
Let $D = \mathsf{dist}(f|_{L_{N+1}})$, and
let $1 \leq k_1 < \cdots < k_h \leq N$ be the $h \geq \lfloor
\alpha N\rfloor$ levels at which $f$ is
$(\e,\delta)$-inefficient.

Let us consider the first level $k_1$.  Let $S_{k_1}' \subseteq
S_{k_1}$ be a subset of size $|S_{k_1}'| \geq \lfloor \delta
|S_{k_1}| \rfloor$ for which
$$(a,b) \in S_{k_1}' \implies f|_{[a,b]} \textrm{ is not $\e$-efficient at granularity $M$}$$
For any such $(a,b) \in
S_{k_1}'$, we know that
\begin{eqnarray*}
\sum_{j=0}^{M-1} d\left( f\left(a+j M^{-k_1-1}\right),
f\left(a+(j+1)M^{-k_1-1}\right)\right)
&>& (1+\e) d(f(a),f(b)) \\
&\geq & d(f(a),f(b)) + \e \frac{M^{-k_1}}{D}.
\end{eqnarray*}
by the definition of (not being) $\e$-efficient,
and the fact that $d(f(a),f(b)) \geq |a-b|/D$. For all segments $(a,b) \in S_{k_1} \setminus S_{k_1}'$, the triangle inequality yields
\begin{eqnarray*}
\sum_{j=0}^{M-1} d\left( f\left(a+j M^{-k_1-1}\right),
f\left(a+(j+1)M^{-k_1-1}\right)\right)
& \geq &  d(f(a),f(b)) \\
\end{eqnarray*}
By summing the above inequalities over all the segments in $S_{k_1}$, we get
\begin{eqnarray*}
\sum_{(u,v) \in S_{k_1+1}} d\left(f(u),f(v)\right) \geq
\sum_{(a,b) \in S_{k_1}} d(f(a),f(b)) + \frac{\e \delta}{2D},
\label{eq:excess}
\end{eqnarray*}
where the extra factor 2 in the denominator on the RHS just comes
from removing the floor from $|S_{k_1}'| \geq \lfloor \delta
|S_{k_1}| \rfloor$. Similarly, for each of the levels $k_2, \ldots,
k_h$, we will pick up an excess term of $\e\delta/(2D)$.  We
conclude that
$$
1 \geq \sum_{(u,v) \in S_{N+1}} d\left(f(u),f(v)\right) \geq
\frac{\e\delta h}{2D},
$$
where the LHS comes from the fact that $f$ is non-expansive.
Simplifying achieves the desired conclusion.
\end{proof}

\subsection{Differentiation for families of geodesics}

Let $G=(V,E)$ be an unweighted graph,
and let $\mathcal P$ denote a family of geodesics (i.e. shortest-paths) in $G$.
Furthermore, assume that every $\gamma \in \mathcal P$ has length $M^r$ for some $M,r \in \mathbb N$.
Let $f : (V,d_G) \to X$ be a non-expansive map into some metric space $(X,d)$.

For the sake of convenience, we will index the vertices along the paths using numbers from $[0,1]$.
Specifically, we will refer to the $i^{th}$ vertex along the path $\gamma \in \mathcal{P}$ by $\gamma\left(\frac{i}{M^{r}}\right)$.
For indices $a,b$, we will use $\gamma[a,b]$ to denote the sub path starting at $\gamma(a)$ and ending at $\gamma(b)$.
We will also use $f|_{\gamma[a,b]}$ to denote the restriction of $f$ to the path $\gamma[a,b]$.
As earlier, the function $f|_{\gamma[a,b]}$ is said to be $\epsilon$-efficient at granularity $M$ if
\begin{eqnarray*}
\sum_{j=0}^{M-1} d\left(\vphantom{\bigoplus} f\left(\vphantom{\bigoplus}\gamma\left(a + M^{-1} (b-a)j\right)\right), f\left(\vphantom{\bigoplus}\gamma\left(a + M^{-1} (b-a)(j+1)\right)\right)\right)
\leq \left(1+\e\right)\, d(f(a),f(b)).
\end{eqnarray*}

Let the sets $L_k$ and $S_k$  be defined as before.
Thus a level-$k$ segment of a path $\gamma \in \mathcal{P}$ is $\gamma[a,b]$ for some $(a,b) \in S_{k}$.
We say that $f$ is {\em $(\epsilon,\delta)$ inefficient at level $k$ for the family of paths $\mathcal{P}$} if
the following holds:
\begin{eqnarray*}
\left | \left \{ (a,b) \in S_k, \gamma \in \mathcal{P} : \textrm{$f|_{\gamma[a,b]}$ is not $\e$-efficient at
granularity $M$} \right\}\right| \geq \delta M^k |\mathcal{P}|.
\end{eqnarray*}

A straightforward variation of the proof of Theorem \ref{thm:coarsediff} yields the following.

\begin{theorem}\label{thm:family}
If a non-expansive map $f : V \to X$ is $(\e,\delta)$-inefficient at an $\alpha$-fraction of levels
$k = 1, 2, \ldots, N$, then $\mathsf{dist}(f) \geq \frac{1}{2} \varepsilon \alpha \delta N$.
\end{theorem}

\begin{proof}
Let $D = \mathsf{dist}(f)$, and
let $1 \leq k_1 < \cdots < k_h \leq N$ be the $h \geq \lfloor
\alpha N\rfloor$ levels for which $f$ is
$(\e,\delta)$-inefficient at level $k_i$.

Let us consider the first level $k_1$.  Let $S_{k_1}' \subseteq
\mathcal{P} \times S_{k_1}$ be a subset of size $|S_{k_1}'| \geq \lfloor \delta
|S_{k_1}| |\mathcal{P}| \rfloor$ for which
$$\big(\gamma, (a,b)\big) \in S_{k_1}' \implies f|_{\gamma[a,b]} \textrm{ is not $\e$-efficient at granularity $M$}.$$
For any such $\big(\gamma,(a,b)\big) \in
S_{k_1}'$, we know that
\begin{eqnarray*}
\sum_{j=0}^{M-1} d\left( f\left(\gamma(a+j M^{-k_1-1})\right),
f\left(\gamma(a+(j+1)M^{-k_1-1})\right)\right)
&>& (1+\e)\, d(f(\gamma(a)),f(\gamma(b))) \\
&\geq & d\left(f(\gamma(a)),f(\gamma(b))\right) + \e \frac{M^{N-k_1}}{D}.
\end{eqnarray*}
by the definition of (not being) $\e$-efficient,
and the fact that $d(f(\gamma(a)),f(\gamma(b))) \geq M^{N}|a-b|/D$. In particular,
summing both sides over all the segments $\gamma[a,b]$ over all paths $\gamma$ and segments $[a,b] \in S_{k_1}$ (and
replacing the preceding inequality by the triangle inequality if
$(a,b) \notin S_{k_1}'$), we get
\begin{eqnarray*}
\sum_{\gamma \in \mathcal{P}} \sum_{(u,v) \in S_{k_1+1}} d\left(f(\gamma(u)),f(\gamma(v))\right) \geq
\sum_{\gamma \in \mathcal{P}}\sum_{(a,b) \in S_{k_1}} d(f(\gamma(a)),f(\gamma(b))) + \frac{ \e \delta M^{N}|\mathcal{P}| }{2D},
\end{eqnarray*}
Similarly, for each of the levels $k_2, \ldots, k_h$, we will pick up an excess term of $\e\delta  M^{N}|\mathcal{P}|/(2D)$.  We
conclude that
$$
M^{N}|\mathcal{P}| \geq \sum_{\gamma \in \mathcal{P}}\sum_{(u,v) \in S_{N+1}} d\left(f(\gamma(u)),f(\gamma(v))\right) \geq \frac{\e\delta h M^{N}|\mathcal{P}|}{2D},
$$
The desired conclusion follows.
\end{proof}

\subsection{Efficient $L_1$-valued maps and monotone cuts}\label{sec:mono}

Finally, we relate monotonicity of $L_1$-valued mappings to properties of their
cut decompositions.

\begin{definition}
A sequence $P = \{x_1,x_2,\ldots,x_k\} \subseteq X$ is said to be {\em monotone with respect to a
cut $(S,\overline{S})$} (where $X = S \uplus \bar S$) if
$S \cap P = \{x_1,x_2,\ldots,x_i\}$ or
$\bar S \cap P = \{x_1,x_2,\ldots,x_i\}$
for some $1 \leq i \leq k$.
\end{definition}

If $\mu$ is a cut measure on a finite set $X$ and $x,y \in X$, we define the {\em separation
measure $\mu^{x|y}$} as follows:  For every $S \subseteq X$,
let $\mu^{x|y}(S) = \mu(S) |\1_S(x)-\1_S(y)|$.

\begin{lemma}\label{lem:cutmonotone}
Let $(X,d)$ be a finite metric space, and let $P = \{x_1, x_2, \ldots, x_k \} \subseteq X$ be a finite sequence.
Given a mapping $f : X \to L_1$, let $\mu$ be the corresponding cut measure (see \eqref{eq:cutmeasure}).
If $f$ is $\epsilon$-efficient on $P$, then
$$
\mu^{x_1|x_k}\left(\left\{ S : \textrm{$P$ is monotone with respect to $(S,\bar S)$}\right\}\right) \geq (1-\epsilon) \|f(x_1)-f(x_k)\|_1.
$$
\end{lemma}

\begin{proof}
If the sequence $P$ is not monotone with respect to a cut $(S,\overline{S})$, then
$$\sum_{i=1}^{k-1} |\1_{S}(x_i) - \1_{S}(x_{i+1})| \geq 2 |\1_{S}(x_1) - \1_{S}(x_{k})|.$$
Now, let $\mathcal E = \{ S : \textrm{$P$ is not monotone with respect to $(S,\bar S)$} \}$,
and for the sake of contradiction, assume that $\mu^{x_1|x_k}(\mathcal E) > \e \|f(x_1)-f(x_k)\|_1$, then
\begin{eqnarray*}
\sum_{i=1}^{k-1} \|f(x_i) - f(x_{i+1})\|_{1} &=&
\sum_{i=1}^{k-1} \left[\int_{\mathcal E} |\1_S(x_i)-\1_S(x_{i+1})|\,d\mu(S)
+ \int_{\bar{\mathcal{E}}} |\1_S(x_i)-\1_S(x_{i+1})|\,d\mu(S)
\right] \\
&\geq &
 2 \int_{\mathcal E} |\1_S(x_1)-\1_S(x_{k})|\,d\mu(S)
+ \int_{\bar{\mathcal{E}}}  |\1_S(x_1)-\1_S(x_{k})|\,d\mu(S) \\
&=&
2 \mu^{x_1|x_k}(\mathcal E) + \mu^{x_1|x_k}(\bar{\mathcal{E}})
\\
& > &
(1+\epsilon) \|f(x_1)-f(x_k)\|_1,
\end{eqnarray*}
where we observe that $\|f(x_1)-f(x_k)\|_1 = \mu^{x_1|x_k}(\mathcal E) + \mu^{x_1|x_k}(\bar{\mathcal{E}}).$
This is a contradiction, since $f$ is assumed to be $\epsilon$-efficient on $P$.
\end{proof}

\section{The distortion lower bound} \label{sec:lb}

Our lower bound examples are the recursively defined family of
graphs $\{K_{2,n}^{\oslash k}\}_{k=1}^{\infty}$.  We recall that
the graphs $K_{2,2}^{\oslash k}$ are known as diamond graphs
\cite{GNRS99,NR03}.

\begin{lemma}\label{lem:stgraph}
Let $G$ be an $s$-$t$ graph with a uniform length function, i.e. $\len(e)=1$ for every $e \in E(G)$.
Then for every $\epsilon,D > 0$, there exists an integer $N=N(G,\e,D)$ such that the following
holds: For any non-expansive map $f : G^{\oslash N} \rightarrow X$ with $\dist(f) \leq D$,
there exists a copy $G'$ of $G$ in $G^{\oslash N}$ such that $f$ is $\epsilon$-efficient on all
$s$-$t$ geodesics in $G'$.
\end{lemma}
\begin{proof}
Let $M=d_G(s,t)$, and
let $\mathcal{P}_G$ denote the family of $s$-$t$ geodesics in $G$.
 Fix $\delta = \frac{1}{|\mathcal{P}_G|}$, $\alpha = \frac{1}{2}$ and $N = \frac{8D}{\epsilon\delta}$.

Let $\mathcal{P}$ denote the family of all $s$-$t$ geodesics in $G^{\oslash N}$.
Each path in $\mathcal{P}$ is of length $M$ and consists of $M^{N}$ edges.  From the choice of parameters,
observe that $\frac{1}{2}\epsilon\alpha\delta N > D$.
Applying Theorem \ref{thm:family} to the family $\mathcal{P}$,
any non-expansive map $f$ with $\dist(f) \leq D$ is $(\epsilon, \delta)$-efficient
at an $\alpha = \frac{1}{2}$-fraction of levels $k = 1,2,\ldots N$.  Specifically,
there exists a level $k$ such that $f$ is $(\epsilon, \delta)$-efficient at level $k$.

For a uniformly random choice of path $\gamma \in \mathcal P$,
and level-$k$ segment $(a,b)$ of $\gamma$,
$f|_{\gamma[a,b]}$ is not $\epsilon$-efficient at granularity $M$
with probability at most $\delta$.
In case of the family $\mathcal{P}$, each of the level-$k$ segments
is nothing more than an $s$-$t$ geodesic in a level-$k$ copy of $G$.

If, for at least one of the level-$k$ copies of $G$, $f$ is $\epsilon$-efficient
on all the $s$-$t$
geodesics in that copy, the proof is complete.  On the contrary,
suppose each level-$k$ copy has an $s$-$t$ geodesic on which $f$ is not $\epsilon$-efficient.
Then in each level-$k$ copy at least a $\delta = \frac{1}{\mathcal{P}_G}$-fraction of the $s$-$t$
geodesics are $\epsilon$-inefficient.  As the level-$k$ copies partition the set of all level-$k$ segments,
this implies that at least a $\delta$-fraction of the segments are $\epsilon$-inefficient.
This contradicts the fact that $f$ is $(\epsilon,\delta)$-efficient at level $k$.
\end{proof}

Although we will not need it, the same type of argument proves the following generalization
to weighted graphs $G$.  The idea is that in $G^{\oslash N}$ for $N$ large enough,
there exists a copy of a subdivision of $G$ with each edge finitely subdivided.
Paying small distortion, we can approximate $G$ (up to uniform scaling)
by this subdivided copy, where the latter is equipped with uniform edge lengths.

\begin{lemma}
Let $G$ be an $s$-$t$ graph with with arbitrary non-negative edge lengths $\len : E(G) \to \mathbb R_+$.
Then for every $\epsilon,D > 0$, there exists an integer $N=N(G,\e,D,\len)$ such that the following
holds: For any non-expansive map $f : G^{\oslash N} \rightarrow X$ with $\dist(f) \leq D$,
there exists a copy $G'$ of $G$ in $G^{\oslash N}$ such that $f$ is $\epsilon$-efficient on all
$s$-$t$ geodesics in $G'$.
\end{lemma}

    In the graph $K_{2,n}$, we will refer to the $n$ vertices other than $s,t$ by $M = \{m_i\}_{i=1}^n$.

\begin{lemma}\label{lem:ktwon}
For $\epsilon < \frac{1}{2}$ and any function $f : V(K_{2,n}) \rightarrow L_1$ that is $\epsilon/n$-efficient
with respect to each of the geodesics
$s$-$m_i$-$t$, for $1 \leq i \leq n$, we have $\dist(f) \geq 2 - \frac{2}{n} - 2\epsilon$.
\end{lemma}

\begin{proof}
Let $\mu$ be the cut measure corresponding to $f$.
By scaling, we may assume that $$\|f(s)-f(t)\|_1 = \mu\left\{ S : \1_S(s) \neq \1_S(t) \right\} = 1.$$
Let $V = V(K_{2,n})$.
Without loss of generality, we assume that $\mu$ is supported on $2^V \setminus \{\emptyset, V\}$.
Let $\gamma_i$ be the geodesic $s$-$m_i$-$t_i$ for $i \in \{1,2,\ldots, n\}$.
Define
$$\mathcal E = \left\{ S : (S, \bar S) \textrm{ is {\em not} monotone with respect to $\gamma_i$ for some $i \in [n]$} \right\}.$$
     Applying Lemma \ref{lem:cutmonotone}, by a union bound and the fact that $f$ is $\epsilon/n$ efficient on every $\gamma_i$,
we see that $\mu(\mathcal E) \leq \epsilon$.

    Consider a cut $(S,\overline{S})$ that is monotone with respect to all the $\gamma_i$ geodesics, and
such that $\mu(S) > 0$.
Let us refer to these cuts as {\it good} cuts.
By monotonicity, and the fact that $S \notin \{0,V\}$, we know that $|\1_S(s)-\1_S(t)|=1$.
Thus for a good cut $(S,\overline{S})$, we have
\begin{equation}\label{eq:distorted}
\sum_{i,j \in [n]} |\1_{S}(m_i) - \1_{S}(m_j)| = 2(|S|-1)(n-|S|-1) \leq  \frac{n^2}{2}.
\end{equation}

It follows that,
    \begin{eqnarray*}
    \sum_{i,j \in [n]} \|f(m_i) - f(m_j) \|_{1} & = &
    \int_{\mathcal E} \sum_{i,j \in [n]} |\1_S(m_i)-\1_S(m_j)|\,d\mu(S) +
    \int_{\bar{\mathcal{E}}} \sum_{i,j \in [n]} |\1_S(m_i)-\1_S(m_j)|\,d\mu(S) \\
          & \leq & \mu(\bar{\mathcal{E}}) \frac{n^2}{2}  + \mu(\mathcal E) n^2 \\
&\leq & (1-\epsilon) \frac{n^2}{2} + \epsilon n^2 \\
&=& \frac{(1+\epsilon)n^2}{2} \|f(s)-f(t)\|_1.
    \end{eqnarray*}
where in the first inequality, we have used \eqref{eq:distorted}, and we recall
that $\|f(s)-f(t)\|_1 = 1$.

Contrasting this with the fact that
$$
\sum_{i,j \in [n]} d_{K_{2,n}}(m_i,m_j) = n(n-1)\, d_{K_{2,n}}(s,t)
$$
yields
$$
\dist(f)
\geq   \frac{n(n-1)}{\frac{(1+\epsilon)n^2}{2}} = \frac{2}{1+\epsilon}\left(1 - \frac{1}{n}\right)
                      \geq  2 - \frac{2}{n} - 2 \epsilon.
$$
\end{proof}

\begin{theorem} \label{thm:mainLB}
For any $n \geq 2$, $\lim_{k \rightarrow \infty} c_1(K_{2,n}^{\oslash k}) \geq 2 - \frac{2}{n}$.
\end{theorem}
\begin{proof}
For any $\epsilon' > 0$, let $N$ be the integer obtained by applying Lemma \ref{lem:stgraph} to
$K_{2,n}$ with $\epsilon = \epsilon'/n, D = 2$ and $G = K_{2,n}$.
 We will show that for any map $f : K_{2,n}^{\oslash N} \rightarrow L_1$,  $\dist(f) \geq 2-\frac{2}{n}-2\epsilon'$.
Without loss of generality, assume that $f$ is non-expansive.
If $\dist(f) \leq 2$, then from Lemma \ref{lem:stgraph} there exists a copy
of $K_{2,n}$ in which $f$ is $\frac{\epsilon'}{n}$ on all the $s$-$t$ geodesics.
Using Lemma \ref{lem:ktwon}, we see that on {\em this}  copy of $K_{2,n}$ we get $\dist(f|_{K_{2,n}}) \geq 2-\frac{2}{n} - 2\epsilon'$.
The result follows by taking $\epsilon' \to 0$.
\end{proof}

\section{Embeddings of $\KK{n}{k}$}\label{sec:embeddings}

In this section, we show that for every fixed $n$, $\lim_{k \to \infty} c_1(\KK{n}{k}) < 2$.

\medskip
\noindent {\bf A next-embedding operator.} Let $T$ be a random
variable ranging over subsets of $V(\KK{n}{k})$, and let $S$ be a
random variable ranging over subsets of $V(K_{2,n})$. We define a
random subset $P_S(T) \subseteq V(\KK{n}{k+1})$ as follows. One
moves from $\KK{n}{k}$ to $\KK{n}{k+1}$ by replacing every edge
$(x,y) \in E(\KK{n}{k})$ with a copy of $K_{2,n}$ which we will call
$K_{2,n}^{(x,y)}$.  For every edge $(x,y) \in \KK{n}{k}$, let
$S^{(x,y)}$ be an independent copy of the cut $S$ (which ranges over
subsets of $V(K_{2,n})$).  We form the cut $P_S(T) \subseteq
V(\KK{n}{k+1})$ as follows. If $(x,y) \in E(\KK{n}{k})$, then for $v
\in V(K_{2,n}^{(x,y)})$, we put
$$
\1_{P_S(T)}(v) =
\begin{cases}
\1_{P_S(T)}\left(s(K_{2,n}^{(x,y)})\right) & \textrm{if } \1_{S^{(x,y)}}(v) = \1_{S^{(x,y)}}\left(s(K_{2,n}^{(x,y)})\right) \\
\1_{P_S(T)}\left(t(K_{2,n}^{(x,y)})\right) & \textrm{otherwise}
\end{cases}
$$
We note that, strictly speaking, the operator $P_S$ depends on $n$
and $k$, but we allow these to be implicit parameters.

\subsection{Embeddings for small $n$}
\label{sec:smalln}

Consider the graph $K_{2,n}$ with vertex set $V =
\{s,t\} \cup M$.  An embedding in the style of \cite{GNRS99} would
define a random subset $S \subseteq V$ by selecting $M' \subseteq M$
to contain each vertex from $M$ independently with probability
$\frac12$, and then setting $S = \{s\} \cup M'$.  The resulting
embedding has distortion 2 since, for every pair $x,y \in M$, we
have $\Pr[\1_S(x) \neq \1_S(y)] = \frac12$.  To do slightly better,
we choose a uniformly random subset $M' \subseteq M$ of size
$\lfloor \frac{n}{2}\rfloor$ and set $S = \{s\} \cup M'$ or $S =
\{s\} \cup (M\setminus M')$ each with probability half. In this
case, we have
$$\Pr[\1_S(x) \neq \1_S(y)] = \frac{ \lfloor \frac{n}{2} \rfloor \cdot \lfloor \frac{n+1}{2} \rfloor}{{n \choose 2}} > \frac12,$$
resulting in a distortion slightly better than 2. A recursive
application of these ideas results in $\lim_{k \to \infty}
c_1(\KK{n}{k}) < 2$ for every $n \geq 1$, though the calculation is
complicated by the fact that the worst distortion is incurred for a
pair $\{x,y\}$ with $x \in M(H)$ and $y \in M(G)$ where $H$ is a
copy of $\KK{n}{k_1}$ and $G$ is a copy of $\KK{n}{k_2}$, and the
relationship between $k_1$ and $k_2$ depends on $n$. (For instance,
$c_1(K_{2,2}) = 1$ while $\lim_{k \to \infty}(\KK{2}{k}) =
\frac{4}{3}$.)

\begin{theorem}\label{thm:smalln}
For any $n,k\in \mathbb N$, we have $c_1(\KK{n}{k}) \leq 2 -
\frac{2}{2\lceil \frac{n}{2} \rceil +1}$.
\end{theorem}

\begin{proof}For simplicity, we prove the bound for $\KK{2n}{k}$.  A similar
analysis holds for $\KK{2n+1}{k}$. We define a random cut $S_k
\subseteq V(\KK{2n}{k})$ inductively. For $k=1$, choose a uniformly
random partition $M(\KK{2n}{1}) = M_s \cup M_t$ with $|M_s| = |M_t|
= n$, and let $S_1 = \{s(\KK{2n}{1})\} \cup \{M_s\}$. The key fact
which causes the distortion to be less than 2 is the following: For
any $x,y \in M(\KK{2n}{1})$, we have \begin{equation}
\label{eqn:keyfact} \Pr[\1_{S_1}(x) \neq \1_{S_1}(y)] =
\frac{n^2}{{{2n} \choose 2}} = \frac{n}{2n-1} >
\frac12.\end{equation} This follows because there are ${{2n} \choose
2}$ pairs $\{x,y\} \in M(\KK{2n}{1})$ and $n^2$ are separated by
$S_1$.

Assume now that we have a random subset $S_k \subseteq
V(\KK{2n}{k})$. We set $S_{k+1} = P_{S_1}(S_k)$ where $P_{S_1}$ is
the operator defined above, which maps random subsets of
$V(\KK{2n}{k})$ to random subsets of $V(\KK{2n}{k+1})$.  In other
words $S_k = P^{k-1}_{S_1}(S_1)$. \remove{We go from $\KK{2n}{k}$ to
$\KK{2n}{k+1}$ by replacing every edge $(x,y) \in E(\KK{2n}{k})$
with a copy of $K_{2,2n}$ which we call $K_{2,2n}^{(x,y)}$. For
every edge $(x,y) \in \KK{2n}{k}$, let $S_1^{(x,y)} \subseteq
V(K_{2,2n}^{(x,y)})$ be an independent copy of the cut $S_1$ already
defined. We form the cut $S_{k+1} \subseteq V(\KK{2n}{k+1})$ as
follows.  If $(x,y) \in E(\KK{2n}{k})$, then for $v \in
V(K_{2,2n}^{(x,y)})$, we put
$$
\1_{S_{k+1}}(v) =
\begin{cases}
\1_{S_k}\left(s(K_{2,2n}^{(x,y)})\right) & \textrm{if } \1_{S_1^{(x,y)}}(v) = \1_{S_1^{(x,y)}}\left(s(K_{2,2n}^{(x,y)})\right) \\
\1_{S_k}\left(t(K_{2,2n}^{(x,y)})\right) & \textrm{otherwise}
\end{cases}
$$}

Let $s_0 = s(\KK{2n}{k})$ and $t_0 = t(\KK{2n}{k})$. It is easy to
see that the cut $S = S_k$ defined above is always monotone with
respect to every $s_0$-$t_0$ shortest path in $\KK{2n}{k}$, thus
every such path has exactly one edge cut by $S_k$, and furthermore
the cut edge is uniformly chosen from along the path, i.e.
$\Pr[\1_{S}(x) \neq \1_{S}(y)] = 2^{-k}$ for every $(x,y) \in
E(\KK{2n}{k})$.  In particular, it follows that if $u,v \in
V(\KK{2n}{k})$ lie along the same simple $s_0$-$t_0$ path, then
$\Pr[\1_S(u) \neq \1_S(v)] = 2^{-k} d(u,v)$.

\begin{figure}
\label{fig:cases} \centering \subfigure[Case I]{
\includegraphics[width = 6cm]{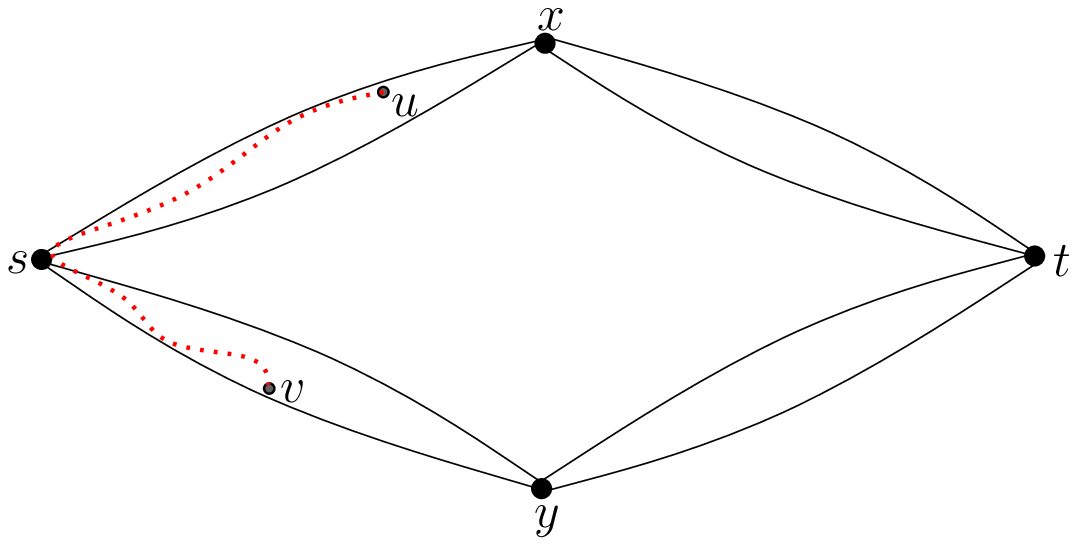}
} \subfigure[Case II]{
\includegraphics[width = 6cm]{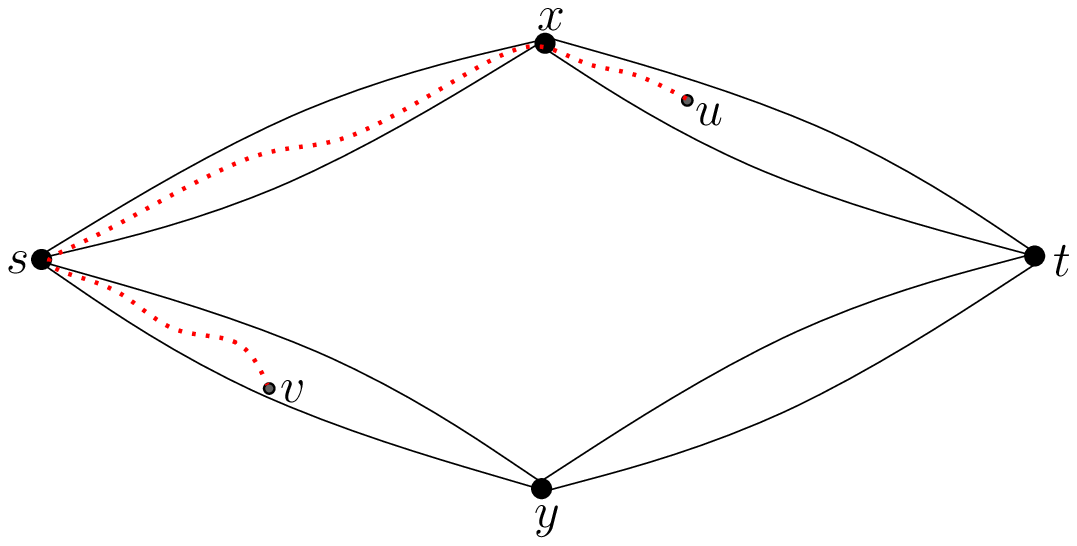}
} \caption{The two cases of Theorem \ref{thm:smalln}} \label{cases}
\end{figure}

Now consider any $u,v \in V(\KK{2n}{k})$.  Fix some shortest path
$P$ from $u$ to $v$. By symmetry, we may assume that $P$ goes left
(toward $s_0$) and then right (toward $t_0$).  Let $s$ be the
left-most point of $P$.  In this case, $s = s(H)$ for some subgraph
$H$ which is a copy of $\KK{2n}{k'}$ with $k' \leq k$, and such that
$u,v \in V(H)$; we let $t = t(H)$. We also have $d(u,v) = d(u,s) +
d(s,v)$. Let $M = M(H)$, and fix $x, y \in M$ which lie along the
$s$-$u$-$t$ and $s$-$v$-$t$ shortest-paths, respectively. Without
loss of generality, we may assume that $d(s,v) \leq d(s,y)$. We need
to consider two cases (see Figure \ref{fig:cases}).

\medskip
\noindent {\bf Case I:} $d(u,s) \leq d(x,s)$.

\medskip

For any pair $a,b \in V(\KK{2n}{k})$, we let $\mathcal E_{a,b}$ be
the event $\{\1_S(a) \neq \1_S(b)\}$. In this case, we have
$\Pr[\mathcal E_{u,v}] = \Pr[\mathcal E_{s,t}] \cdot \Pr[\mathcal
E_{u,v}\,|\,\mathcal E_{s,t}].$ Since $s,t$ clearly lie on a
shortest $s_0$-$t_0$ path, we have $\Pr[\mathcal E_{s,t}] = 2^{-k}
d(s,t)$. For any event $\EE$, we let $\mu[\EE] =
\Pr[\EE\,|\,\EE_{s,t}]$. Now we calculate using \eqref{eqn:keyfact},
\begin{eqnarray*}
\mu[\EE_{u,v}] &\geq & \mu[\EE_{x,y}] \left( \mu[\EE_{x,s}\sep
\EE_{x,y}] \mu[\EE_{u,s}\sep \EE_{x,s}, \EE_{x,y}]
+ \mu[\EE_{x,t}\sep \EE_{x,y}] \mu[\EE_{v,s}\sep \EE_{x,t},\EE_{x,y}] \right)\\
&=&
\frac{n}{2n-1} \left( \frac{1}{2} \cdot \frac{d(u,s)}{d(x,s)} + \frac{1}{2} \cdot \frac{d(v,s)}{d(y,s)} \right) \\
&=& \frac{n}{2n-1} \frac{d(u,v)}{d(s,t)}.
\end{eqnarray*}

Hence in this case, $\Pr[\1_S(u) \neq \1_S(v)] \geq \frac{n}{2n-1}
\cdot 2^{-k} d(u,v)$.

\medskip
\noindent {\bf Case II:} $d(u,s) \geq d(x,s)$.

\medskip

Here, we need to be more careful about bounding $\mu[\EE_{u,v}]$. It
will be helpful to introduce the notation $a \mapsto b$ to represent
the event $\{\1_S(a) = \1_S(b)\}$.  We have,
\begin{eqnarray*}
\mu[\EE_{u,v}] &=& \mu[x \mapsto t, y \mapsto s] + \mu[x \mapsto t, y
\mapsto t, v \mapsto s] +
\mu[x \mapsto s, y \mapsto s, u \mapsto t] \\
&& +\,\, \mu[x \mapsto s, y \mapsto t, u \mapsto t, v \mapsto s] +
\mu[x \mapsto s, y \mapsto t, u \mapsto s, v \mapsto t] \\
&=& \frac{1}{2} \frac{n}{2n-1} + \frac{n-1}{2n-1}
\frac{d(v,y)+d(u,x)}{d(s,t)} + \frac{1}{2} \frac{n}{2n-1}
\left(\frac{d(u,x)d(v,y)+d(u,t)d(v,s)}{d(x,t)d(y,s)} \right)
\end{eqnarray*}
If we set $A = \frac{d(v,s)}{d(s,t)}$ and $B =
\frac{d(u,x)}{d(s,t)}$, then $\frac{d(u,v)}{d(s,t)} = \frac12 + A +
B$ and simplifying the expression above, we have
\begin{eqnarray*}
\mu[\EE_{u,v}] &=& \frac12 + B + \frac{A}{2n-1} - \frac{4n}{2n-1} AB
\label{eq:fourthirds}
\end{eqnarray*}
Since the shortest path from $u$ to $v$ goes through $s$ by
assumption, we must have $A+B \leq \frac12$.  Thus we are interested
in the minimum of $\mu[\EE_{u,v}]/(\frac12 + A + B)$ subject to the
constraint $A+B \leq \frac12$. It is easy to see that the minimum is
achieved at $A+B=\frac12$, thus setting $B = \frac12-A$, we are left
to find
$$
\min_{0 \leq A \leq \frac12} \left\{1 - 2A +
\frac{4nA^2}{2n-1}\right\} = \frac{2n+1}{4n}.
$$
(The minimum occurs at $A = \frac12 - \frac{1}{4n}$.) So in this
case, $\Pr[\1_S(u) \neq \1_S(v)] \geq \frac{2n+1}{4n} 2^{-k}
d(u,v)$.

\medskip

Combining the above two cases, we conclude that the distribution $S
= S_k$ induces an $L_1$ embedding of $\KK{2n}{k}$ with distortion at
most $\max\{\frac{2n-1}{n}, \frac{4n}{2n+1}\} = 2 - \frac{2}{2n+1}$.
A similar calculation yields $$c_1(\KK{2n+1}{k} ) \leq \left(
\min_{0 \leq A \leq \frac12} \left\{1 - 2A +
\frac{4(n+1)A^2}{2n+1}\right\}\right)^{-1} = 2 - \frac{2}{2n+3}.$$

\remove{ Now, from \eqref{eqn:keyfact}, we know that $\mu[\1_S(x)
\neq \1_S(y)] = \frac{n}{2n-1}$. The rest of the computations follow
easily by symmetry and independence, and we conclude that
\begin{eqnarray*}
\mu[\1_{S}(u) \neq \1_S(v)] &=& \frac{n}{2n-1} \cdot \frac{1}{2}
\left(\frac{d(u,s)}{d(s,t)} + \frac{d(v,s)}{d(s,t)}\right)+
\end{eqnarray*}
}
\end{proof}

\subsection*{Acknowledgments}

We thank Alex Eskin for explaining coarse differentiation to us,
Yuri Rabinovich for relating the distortion $2$ conjecture,
and Kostya and Yury Makayrchev for a number of fruitful discussions.

\bibliographystyle{abbrv}
\bibliography{trees,gl}

\def\cprime{$'$}
\begin{thebibliography}{10}

\bibitem{ADGIR03}
A.~Andoni, M.~Deza, A.~Gupta, P.~Indyk, and S.~Raskhodnikova.
\newblock Lower bounds for embedding of edit distance into normed spaces.
\newblock In {\em Proceedings of the 14th annual ACM-SIAM Symposium on Discrete
  Algorithms}, 2003.

\bibitem{ALN05}
S.~Arora, J.~R. Lee, and A.~Naor.
\newblock Euclidean distortion and the {S}parsest {C}ut.
\newblock {\em J. Amer. Math. Soc.}, 21(1):1--21, 2008.

\bibitem{ARV04}
S.~Arora, S.~Rao, and U.~Vazirani.
\newblock Expander flows, geometric embeddings, and graph partitionings.
\newblock In {\em 36th Annual Symposium on the Theory of Computing}, pages
  222--231, 2004.
\newblock To appear, {\em J. ACM}.

\bibitem{AR98}
Y.~Aumann and Y.~Rabani.
\newblock An {$O(\log k)$} approximate min-cut max-flow theorem and
  approximation algorithm.
\newblock {\em SIAM J. Comput.}, 27(1):291--301 (electronic), 1998.

\bibitem{benlin}
Y.~Benyamini and J.~Lindenstrauss.
\newblock {\em Geometric nonlinear functional analysis. {V}ol. 1}, volume~48 of
  {\em American Mathematical Society Colloquium Publications}.
\newblock American Mathematical Society, Providence, RI, 2000.

\bibitem{Bourgain85}
J.~Bourgain.
\newblock On {L}ipschitz embedding of finite metric spaces in {H}ilbert space.
\newblock {\em Israel J. Math.}, 52(1-2):46--52, 1985.

\bibitem{BKL06}
B.~Brinkman, A.~Karagiozova, and J.~R. Lee.
\newblock Vertex cuts, random walks, and dimension reduction in series-parallel
  graphs.
\newblock In {\em 39th Annual Symposium on the Theory of Computing}, pages
  621--630, 2007.

\bibitem{CLV07}
A.~Chakrabarti, A.~Jaffe, J.~R. Lee, and J.~Vincent.
\newblock Embeddings, flows, and cuts in 2-sums of graphs.
\newblock Submitted, 2008.

\bibitem{Cheeger99}
J.~Cheeger.
\newblock Differentiability of {L}ipschitz functions on metric measure spaces.
\newblock {\em Geom. Funct. Anal.}, 9(3):428--517, 1999.

\bibitem{CK1-06}
J.~Cheeger and B.~Kleiner.
\newblock Differentiating maps into ${L}^1$ and the geometry of {B}{V}
  functions.
\newblock arXiv:math.MG/0611954, 2006.

\bibitem{CKpub}
J.~Cheeger and B.~Kleiner.
\newblock Generalized differential and bi-{L}ipschitz nonembedding in {$L\sp
  1$}.
\newblock {\em C. R. Math. Acad. Sci. Paris}, 343(5):297--301, 2006.

\bibitem{CK-Chern}
J.~Cheeger and B.~Kleiner.
\newblock On the differentiability of {L}ipschitz maps from metric measure
  spaces to {B}anach spaces.
\newblock In {\em Inspired by {S}. {S}. {C}hern}, volume~11 of {\em Nankai
  Tracts Math.}, pages 129--152. World Sci. Publ., Hackensack, NJ, 2006.

\bibitem{CK2009}
J.~Cheeger and B.~Kleiner.
\newblock Metric differentiation, monotonicity and maps to ${L}^1$.
\newblock arXiv:0907.3295, 2009.

\bibitem{CGNRS06}
C.~Chekuri, A.~Gupta, I.~Newman, Y.~Rabinovich, and A.~Sinclair.
\newblock Embedding {$k$}-outerplanar graphs into {$l\sb 1$}.
\newblock {\em SIAM J. Discrete Math.}, 20(1):119--136, 2006.

\bibitem{GeoCuts}
M.~M. Deza and M.~Laurent.
\newblock {\em Geometry of cuts and metrics}, volume~15 of {\em Algorithms and
  Combinatorics}.
\newblock Springer-Verlag, Berlin, 1997.

\bibitem{DiestelBook}
R.~Diestel.
\newblock {\em Graph theory}, volume 173 of {\em Graduate Texts in
  Mathematics}.
\newblock Springer-Verlag, Berlin, third edition, 2005.

\bibitem{EFW06}
A.~Eskin, D.~Fisher, and K.~Whyte.
\newblock Quasi-isometries and rigidity of solvable groups.
\newblock Preprint, 2006.

\bibitem{FSSC01}
B.~Franchi, R.~Serapioni, and F.~Serra~Cassano.
\newblock Rectifiability and perimeter in the {H}eisenberg group.
\newblock {\em Math. Ann.}, 321(3):479--531, 2001.

\bibitem{GNRS99}
A.~Gupta, I.~Newman, Y.~Rabinovich, and A.~Sinclair.
\newblock Cuts, trees and {$l\sb 1$}-embeddings of graphs.
\newblock {\em Combinatorica}, 24(2):233--269, 2004.

\bibitem{Ind01}
P.~Indyk.
\newblock Algorithmic applications of low-distortion geometric embeddings.
\newblock In {\em 42nd Annual Symposium on Foundations of Computer Science},
  pages 10--33. IEEE Computer Society, 2001.

\bibitem{KN06}
S.~Khot and A.~Naor.
\newblock Nonembeddability theorems via {F}ourier analysis.
\newblock {\em Math. Ann.}, 334(4):821--852, 2006.

\bibitem{KV05}
S.~Khot and N.~Vishnoi.
\newblock The unique games conjecture, integrality gap for cut problems and
  embeddability of negative type metrics into $\ell_1$.
\newblock In {\em 46th Annual Symposium on Foundations of Computer Science},
  pages 53--62. IEEE Computer Soc., Los Alamitos, CA, 2005.

\bibitem{K94}
B.~Kirchheim.
\newblock Rectifiable metric spaces: local structure and regularity of the
  {H}ausdorff measure.
\newblock {\em Proc. Amer. Math. Soc.}, 121(1):113--123, 1994.

\bibitem{LN-heisenberg}
J.~R. Lee and A.~Naor.
\newblock ${L}_p$ metrics on the {H}eisenberg group and the {G}oemans-{L}inial
  conjecture.
\newblock In {\em 47th Annual Symposium on Foundations of Computer Science}.
  IEEE Computer Soc., Los Alamitos, CA, 2006.

\bibitem{LinialSurvey}
N.~Linial.
\newblock Finite metric-spaces---combinatorics, geometry and algorithms.
\newblock In {\em Proceedings of the International Congress of Mathematicians,
  Vol. III (Beijing, 2002)}, pages 573--586, Beijing, 2002. Higher Ed. Press.

\bibitem{LLR95}
N.~Linial, E.~London, and Y.~Rabinovich.
\newblock The geometry of graphs and some of its algorithmic applications.
\newblock {\em Combinatorica}, 15(2):215--245, 1995.

\bibitem{MatOpen}
J.~Matou{\v{s}}ek.
\newblock Open problems on low-distortion embeddings of finite metric spaces.
\newblock Online: {\texttt http://kam.mff.cuni.cz/$\sim$matousek/metrop.ps}.

\bibitem{Mat01}
J.~Matou{\v{s}}ek.
\newblock {\em Lectures on discrete geometry}, volume 212 of {\em Graduate
  Texts in Mathematics}.
\newblock Springer-Verlag, New York, 2002.

\bibitem{NR03}
I.~Newman and Y.~Rabinovich.
\newblock A lower bound on the distortion of embedding planar metrics into
  {E}uclidean space.
\newblock {\em Discrete Comput. Geom.}, 29(1):77--81, 2003.

\bibitem{OS81}
H.~Okamura and P.~D. Seymour.
\newblock Multicommodity flows in planar graphs.
\newblock {\em J. Combin. Theory Ser. B}, 31(1):75--81, 1981.

\bibitem{pansu}
P.~Pansu.
\newblock M\'etriques de {C}arnot-{C}arath\'eodory et quasiisom\'etries des
  espaces sym\'etriques de rang un.
\newblock {\em Ann. of Math. (2)}, 129(1):1--60, 1989.

\bibitem{P2001}
S.~D. Pauls.
\newblock The large scale geometry of nilpotent {L}ie groups.
\newblock {\em Comm. Anal. Geom.}, 9(5):951--982, 2001.

\end{thebibliography}

\end{document}